\documentclass[envcountsame]{llncs}
\usepackage[latin1]{inputenc}
\usepackage{longtable,helvet,courier,fancyhdr}
\usepackage{amsfonts}
\usepackage{amsmath,amssymb,bbm}
\usepackage{multicol}
\usepackage{algorithm}
\usepackage{algorithmic}
\algsetup{indent=2em}
\renewcommand{\algorithmicrequire}{\textbf{Input:}}
\renewcommand{\algorithmicensure}{\textbf{Output:}}

\DeclareMathOperator{\lt}{LT}
\DeclareMathOperator{\syz}{Syz}
\DeclareMathOperator{\rep}{Rep}
\DeclareMathOperator{\sig}{Sig}

\DeclareMathOperator{\lm}{LM}
\DeclareMathOperator{\NM}{NM}
\DeclareMathOperator{\lcm}{LCM}
\DeclareMathOperator{\lc}{LC}

\DeclareMathOperator{\spoly}{SPoly}
\DeclareMathOperator{\poly}{Poly}
\DeclareMathOperator{\anc}{Anc}
\DeclareMathOperator{\cls}{cls}
\def\KK{\mathbbm{k}}
\def\I{\mathcal{I}}
\def\ri{\rangle}
\def\li{\langle}
\def\P{\mathcal{P}}
\def\M{\mathcal{M}}

\begin{document}
\title{Computation of Pommaret Bases Using Syzygies}
\author{Bentolhoda Binaei\inst{1} \and Amir Hashemi\inst{1,2} \and
  Werner M. Seiler\inst{3}}
\institute{Department of Mathematical Sciences, 
Isfahan University of Technology\\ Isfahan, 84156-83111, Iran;
\and School of Mathematics, Institute for Research in Fundamental
Sciences (IPM), Tehran, 19395-5746, Iran\\
\email{h.binaei@math.iut.ac.ir}\\
\email{Amir.Hashemi@cc.iut.ac.ir}
\and Institut f\"{u}r Mathematik,
Universit\"at Kassel\\
Heinrich-Plett-Stra\ss e 40, 34132 Kassel, Germany\\
\email{seiler@mathematik.uni-kassel.de}
}
\maketitle

\begin{abstract}
  We investigate the application of syzygies for efficiently computing
  (finite) Pommaret bases. For this purpose, we first describe a
  non-trivial variant of Gerdt's algorithm \cite{gerdt} to construct an
  involutive basis for the input ideal as well as an involutive basis for
  the syzygy module of the output basis.  Then we apply this new algorithm
  in the context of Seiler's method to transform a given ideal into
  quasi stable position to ensure the existence of a finite Pommaret basis
  \cite{Seiler2}.  This new approach allows us to avoid superfluous
  reductions in the iterative computation of Janet bases required by this
  method. We conclude the paper by proposing an involutive variant of the
  signature based algorithm of Gao et al.\ \cite{gvw} to compute
  simultaneously a Gr\"obner basis for a given ideal and for the syzygy
  module of the input basis. All the presented algorithms have been
  implemented in {\sc Maple} and their performance is evaluated via a set
  of benchmark ideals.
\end{abstract}

\section{Introduction}

{\em Gr\"obner bases} provide a powerful computational tool for a wide
variety of problems connected to multivariate polynomial ideals. Together
with the first algorithm to compute them, they were introduced by
Buchberger in his PhD thesis \cite{Bruno1}. Later on, he discovered two
criteria to improve his algorithm \cite{Bruno2} by omitting superfluous
reductions. In 1983, Lazard \cite{Daniel83} developed a new approach by
using linear algebra techniques to compute Gr\"obner bases. In 1988,
Gebauer and M\"oller \cite{gm}, by interpreting Buchberger's criteria in
terms of syzygies, presented an efficient way to improve Buchberger's
algorithm. Furthermore, M\"oller et al.\ \cite{MMT} extended this idea and
described the first {\em signature-based} algorithm to compute Gr\"obner
bases. In 1999, Faug\`{e}re \cite{F4}, by applying fast linear algebra on
sparse matrices, found his F$_4$ algorithm to compute Gr\"{o}bner
bases. Then, he introduced the well-known F$_5$ algorithm \cite{F5} that
uses two new criteria (F$_5$ and {\em IsRewritten}) based on the idea of
signatures and that performs no useless reduction as long as the input
polynomials define a (semi-)regular sequence. Finally, Gao et
al. \cite{gvw} presented a new approach to compute simultaneously Gr\"obner
bases for an ideal and its syzygy module.
 
Involutive bases may be considered as a special kind of non-reduced
Gr\"obner bases with additional combinatorial properties. They originate
from the works of Janet \cite{janet} on the analysis of partial
differential equations. By evolving related methods used by Pommaret
\cite{pommaret}, the notion of {\em involutive polynomial bases} was
introduced by Zharkov and Blinkov \cite{zharkov}. Later, Gerdt and Blinkov
\cite{zbMATH01969239} generalised these ideas to the concepts of {\em
  involutive divisions} and {\em involutive bases} for polynomial ideals to
produce an effective alternative approach to Buchberger's algorithm (for
the efficiency analysis of an implementation of Gerdt's algorithm
\cite{gerdt}, we refer to the web pages {\tt
  http://invo.jinr.ru}). Recently, Gerdt et al.\ \cite{zbMATH06264259}
proposed a signature-based approach to compute involutive bases.

In this article we discuss effective approaches to compute involutive bases
and in particular {\em Pommaret bases}. These bases are a special kind of
involutive bases introduced by Zharkov and Blinkov \cite{zharkov}. While
finite Pommaret bases do not always exist, every ideal in a sufficiently
generic position has one (see \cite{hss:detgen} for an extensive discussion
of this topic). A finite Pommaret basis reflects many (homological)
properties of the ideal it generates. For example, many invariants like
dimension, depth and Castelnuovo-Mumford regularity can be easily read off
from it. We note that all these invariants remain unchanged under
coordinate transformations. We refer to \cite{Seiler} for a comprehensive
overview of the theory and applications of Pommaret bases.

We first propose a variant of Gerdt's algorithm to compute an involutive
basis which simultaneously determines an involutive basis for the syzygy
module of the output basis. Based on it, we improve Seiler's method
\cite{Seiler2} to compute a linear change of coordinates which brings the
input ideal into a generic position so that the new ideal has a finite
Pommaret basis. Then, as a related work, we describe an involutive version
of the approach by Gao et al.\ \cite{gvw} to compute simultaneously
Gr\"obner bases of a given ideal and of the syzygy module of the input
basis. All the algorithm described in this paper have been implemented in
{\sc Maple} and their efficiency is illustrated via a set of benchmark
ideals.

This paper is organized as follows. In Section \ref{Sec:1}, we review basic
definitions and notations related to involutive bases. Section \ref{Sec:2}
is devoted to a variant of Gerdt's algorithm which also computes an
involutive basis for the syzygy module of the output basis. In Section
\ref{Sec:3}, we show how to apply it in the computation of Pommaret
bases. Finally in Section \ref{Sec:4}, we conclude by presenting an involutive variant
of the algorithm of Gao et al.\ by combining it with Gerdt's algorithm.
\section{Preliminaries}
\label{Sec:1}
In this section, we review basic notations and preliminaries needed in the
subsequent sections. Throughout this paper, we assume that
$\P=\KK[x_{1},\ldots,x_{n}]$ is the polynomial ring over an infinite field
$\KK$. We consider polynomials $f_1,\ldots ,f_k\in \P$ and the ideal
$\I=\li f_1,\ldots ,f_k \ri$ generated by them. The total degree and the
degree w.r.t. a variable $ x_{i} $ of a polynomial in $f\in \P$ are denoted
by $ \deg(f) $ and $ \deg_{i}(f) $, respectively. In addition,
$\M=\lbrace x_{1}^{\alpha_{1}}\cdots x_{n}^{\alpha_{n}} \mid \alpha_{i}\geq
0,\,1 \leq i \leq n \rbrace$ stands for the monoid of all monomials in
$\P$. We use throughout the reverse degree lexicographic ordering with
$x_n\prec \cdots \prec x_1$. The leading monomial of a given polynomial
$ f \in \P$ w.r.t. $ \prec $ is denoted by $ \lm(f) $. If $ F \subset \P $
is a finite set of polynomials, $ \lm(F) $ denotes the set
$ \lbrace \lm(f)\ \mid f \in F \rbrace $. The leading coefficient of $ f $,
denoted by $ \lc(f) $, is the coefficient of $ \lm(f) $. The leading term
of $f$ is defined to be $ \lt(f)=\lm(f)\lc(f) $. A finite set
$ G= \lbrace g_{1},\ldots , g_{t} \rbrace \subset \P$ is called a {\em
  Gr\"obner basis} of $ \I $ w.r.t $ \prec $ if
$ \lm(\I)= \langle \lm(g_{1}),\ldots , \lm(g_{t}) \rangle $ where
$ \lm(\I)= \langle \lm(f) \,\,\vert \,\, f \in \I \rangle $. We refer
e.g. to the book of Cox et al. \cite{little} for further details on
Gr\"obner bases.

An analogous notion of Gr\"obner bases may be defined for sub-modules of
$\P^t$ for some $t$, see \cite{cox}. In this direction, let us recall some
basic notations and results. Let $\{ \mathbf{e}_1 ,\ldots,\mathbf{e}_t \}$
be the standard basis of $ \P^t $.  A module monomial in $ \P^t $ is an
element of the form $ x^{\alpha} \mathbf{e}_i $ for some $ i $, where
$ x^{\alpha} $ is a monomial in $\P$. So, each $ f \in \P^t $ can be
written as a $ \KK $-linear combination of module monomials in $ \P^t $.  A
total ordering $<$ on the set of monomials of $ \P^t $ is called a {\em
  module monomial ordering} if the following conditions are satisfied:
\begin{itemize}
\item
 if $\mathbf{m}$ and $\mathbf{n}$ are two module monomials such that
$\mathbf{n} < \mathbf{m}$ and $ x^{\alpha} \in \P$ is a monomial  then $x^{ \alpha} \mathbf{n} < x^{ \beta} \mathbf{m}$,
\item
$ < $ is well-ordering.
\end{itemize}
In addition, we say that $x^{\alpha}\mathbf{e} _i$ divides
$x^\beta\mathbf{e} _j$ if $i = j$ and $x^\alpha$ divides $x^\beta$. Based
on these definitions, one is able to extend the theory of Gr\"obner bases
to sub-modules of the  $\P$-modules of finite rank. Some well-known
examples of module monomial ordering are term over position (TOP), position
over term (POT) and the Schreyer ordering.

\begin{definition}
Let $\{g_1 , \ldots , g_t \} \subset \P$ and $ \prec $ a monomial ordering on $ \P $. We define the {\em Schreyer module ordering} on $\P^t$ as follows: We write $x^{ \alpha}\mathbf{e}_i \prec_{s}  x^{ \beta}\mathbf{e}_j$ if either $\lm(x^{\alpha} g_i ) \prec
\lm(x^{\beta} g_j )$, or $\lm(x^{\alpha} g_i ) = \lm(x^{\beta} g_j )$ and $j < i$.
\end{definition}
Schreyer proposed in his master thesis \cite{schreyer} a slight modification of Buchberger's algorithm to compute a Gr\"obner basis for the syzygy module of a Gr\"obner basis.
\begin{definition} 
Let us consider $ G=(g_{1}, \ldots ,g_{t}) \in \P^{t}$. The (first) {\em syzygy module} of $G$ is defined to be $  \syz(G)= \lbrace (h_{1}, \ldots , h_{t}) \ | \  h_{i} \in \P, \sum _{i=1}^{t}h_{i}g_{i}=0 \rbrace$.
\end{definition} 

Let $G = \lbrace g_{1} , \ldots , g_{t} \rbrace $ be a Gr\"obner basis. By
Buchberger's criterion, each $S$-polynomial has a standard representation:
$\spoly(g_{i},g_{j})=a_{ji}m_{ji}g_{i}-a_{ij}m_{ij}g_{j}=h_{ij1}g_{1}+\cdots
+h_{ijt}g_{t}$ where $a_{ji},a_{ij}\in \KK$, $ h_{ijl}\in \P $ and
$m_{ji},m_{ij}$ are monomials. Let $\mathbf{S}_{ij}=a_{ji}m_{ji}{\bf
  e}_{i}-a_{ij}m_{ij}{\bf e}_{j} -h_{ij1}{\bf e}_{1}-\cdots -h_{ijt}{\bf
  e}_{t}$ be the corresponding syzygy.
\begin{theorem}[Schreyer's Theorem]
With the above introduced notations,  the set $\{\mathbf{S}_{ij}  \ | \ 1\le i<j \le t\}$ is a Gr\"obner basis for  $ \syz(g_{ 1} , \ldots , g_{ t} )$ w.r.t. $\prec_s$.
\end{theorem}
\begin{example}
  Let $ F= \lbrace  xy-x, x^2-y \rbrace  \subset \KK[x,y]  $. The Gr\"obner basis of $ F $ w.r.t. $x\prec_{dlex} y$ is $ {G= \lbrace g_{1}=xy-x,g_{2}=x^2-y,g_{3}=y^2-y \rbrace} $ and the Gr\"obner basis of $\syz(g_1,g_2,g_3)$ is $ \lbrace (x,-y+1,-1),(-x,y^{2}-1,-x^{2}+y+1), (y,0,-x) \rbrace$.
\end{example}
If $F=\{f_1,\ldots ,f_k\}$ is {\em not} a Gr\"obner basis, Wall \cite{Wall}
proposed an effective method to compute $\syz(F)$. If the extended set
$G=f_1,\ldots ,f_k,f_{k+1},\ldots ,f_t$ is a Gr\"obner basis of $\langle F\rangle$, then
$\syz(F)=\{A{\bf s} \ | \  {\bf s}\in \syz(G)\}$ where $A$ is a
matrix such that $G=FA$.

We conclude this section by recalling some definitions and results from the theory of  involutive bases (see \cite{gerdt,Seiler} for more details).  Given a set of polynomials, an involutive division partitions the variables into two disjoint subsets of {\em multiplicative} and {\em non-multiplicative} variables. 
\begin{definition} 
  An {\em involutive division} $ \mathcal{L} $ is given on $\M $ if for any finite set $ U \subset \M $ and any $ u \in U $, the set of variables is partitioned into the subsets of multiplicative variables $ M_{ \mathcal{L}}(u,U) $ and non-multiplicative variables $NM _{\mathcal{L}}(u,U) $ such that the following conditions hold where $ \mathcal{L}(u,U) $ denotes the monoid generated by $M_{ \mathcal{L}}(u,U) $:
\begin{enumerate}
\item
 $ v,u \in U$, $u \mathcal{L}(u,U) \cap v \mathcal{L}(v,U) \neq \emptyset  $ $\Rightarrow$ $ u \in v\mathcal{L}(v,U)$ or $v \in u \mathcal{L}(u,U) $,
\item
$ v \in U$, $ v \in u\mathcal{L}(u,U)$ $ \Rightarrow $ $\mathcal{L}(v,U) \subset \mathcal{L}(u,U) $,
\item
$V \subset U$ and $ u \in V$ $ \Rightarrow $ $ \mathcal{L}(u,U) \subset \mathcal{L}(u,V)$.
\end{enumerate}
We shall write $ u \mid _{\mathcal{L}} w $ if $ w \in u \mathcal{L}(u,U) $. In this case, $ u $ is called an $ \mathcal{L} $-involutive divisor of $ w $ and  $ w $ an $ \mathcal{L} $-involutive multiple of $ u $. 
\end{definition}

We recall the definitions of the Janet and Pommaret division, respectively.
\begin{example}
Let $ U \subset \P$ be a finite set of monomials. For each sequence $ d_{1} , \ldots ,d_{n} $ of non-negative integers and for each $ 1 \le i  \le n $ we define
$$ [d_{1}, \ldots ,d_{i}]= \lbrace  u \in U \ | \ d_{j}= \deg_{j}(u) ,\,\,1 \leq j \leq i \rbrace .$$
The variable $ x_{1} $ is Janet multiplicative (denoted by $\mathcal{J}$-multiplicative) for $ u \in U $ if $ \deg_{1}(u)= \max \lbrace \deg_{1}(v)\,\, \vert \,\,\,v \in U \rbrace $. For $ i > 1 $ the variable $ x_{i} $ is Janet multiplicative for  $u \in [d_{1}, \ldots ,d_{i-1}] $ if $ \deg_{i}(u)= \max \lbrace \deg_{i}(v) \ | \  v \in [d_{1}, \ldots ,d_{i-1}] \rbrace $.
\end{example}

\begin{example}\label{def:red1}
For $ u=x_{1}^{d_{1}} \cdots x_{k}^{d_{k}} $  with $ d_{k}> 0 $ the variables $ \lbrace x_{k}, \ldots ,x_{n}  \rbrace $ are considered as Pommaret multiplicative (denoted by $\mathcal{P}$-multiplicative) and the other variables as Pommaret non-multiplicative. For $ u = 1 $ all the variables are multiplicative. The integer $k$ is called the {\em class} of $u$ and is denoted by $\cls(u)$.
\end{example}

\begin{definition}
  The set  $ F \subset \P $ is called  {\em involutively head autoreduced} if for each $f\in F$ there is no $h\in F\setminus \{f\}$ with $\lm(h) \mid _{\mathcal{L}} \lm(f)$.
\end{definition}
\begin{definition}\label{def:red}
Let $I \subset \P$ be an ideal and $\mathcal{L}$ an involutive division. An involutively head autoreduced subset $H \subset \I$ is an {\em involutive basis} for $\I$  if for all $f \in \I$ there exists $h\in H$ so that $\lm(h) \mid_{\mathcal{L}} \lm(f)$.
\end{definition}
\begin{example}
  For the ideal $ \I=\langle xy, y^{2} , z \rangle \subset \KK[x,y,z]$ the
  set $ \lbrace xy, y^{2} , z, xz, yz \rbrace $ is a Janet basis, but there
  exists only an infinite Pommaret basis of the form
  $ \lbrace xy, y^{2} , z, xz, yz, x^{2}y, x^{2} z,\ldots , x^{k}
  y,x^{k}z,\ldots \rbrace $. One can show that every ideal has a finite
  Janet basis, i.\,e.\ the Janet division is Noetherian.
\end{example}

Gerdt  \cite{gerdt} proposed an efficient algorithm to construct involutive bases using  a completion process where prolongations of given elements by non-multiplicative variables are reduced. This process terminates in finitely many steps for any Noetherian division. In addition, Seiler  \cite{Seiler2} characterized the ideals having finite Pommaret bases by relating them to the notion of quasi stability. More precisely, a given ideal has a finite Pommaret basis iff it is in {\em quasi stable position} (or equivalently if the coordinates are $\delta$-regular) see \cite[Prop. 4.4]{Seiler2}.
\begin{definition}
A monomial ideal $\I$ is called \emph{quasi stable} if for any monomial $m \in \I$  and all integers $i, j,s$ with $1 \le j < i \le n$ and $s>0$, if $x_i^s\mid m$ there exists an integer $t\ge 0$ such that $x_j^tm/x_i^s\in \I$. A homogeneous ideal $\I$ is in {\em quasi stable position} if $\lm(\I)$ is quasi stable.
\end{definition}

\section{Computation of Involutive Basis for Syzygy Module}
\label{Sec:2}
We present now an effective approach to compute, for a given ideal,
simultaneously involutive bases of the ideal and of its syzygy module. We
first recall some related concepts and facts from \cite{Seiler2}. In
loc. cit., an involutive version of Schreyer's theorem is stated where
$S$-polynomials are replaced by non-multiplicative prolongations and an
involutive normal form algorithm is used.

More precisely, let $H \subset \P^t$ be a finite set for some
$t \in \mathbb{N} $, $ \prec_s$ the corresponding Schreyer ordering and
$\mathcal{L}$ an involutive division. We divide $H$ into $t$ disjoint
subsets
$H_{i} = \{ \mathbf{h} \in H \ | \ \lm(\mathbf{h}) = x^\alpha
\mathbf{e}_{i}, x^\alpha \in \M \}$. In addition, for each $i$, let
$B_{i} = \{x^{\alpha} \in \M \ \vert \ x^{\alpha}\mathbf{e}_{i} \in
\lm(H_{i}) \}$. We assign to each $\mathbf{h} \in H_{i}$ the multiplicative
variables
$M_{\mathcal{L},H,\prec} (\mathbf{h}) = \{ x_i \ | \ x_i \in
M_{\mathcal{L}, B_i} (x^\alpha) \ {\rm with} \
\lm(\mathbf{h})=x^{\alpha}\mathbf{e}_{i}\}$. Then, the definition of
involutive bases for sub-modules proceeds as for ideals.

Let $H = \{h_1,\ldots ,h_t \}\subset \P$ be an involutive basis. Let
$h_{i} \in H$ be an arbitrary element and $x_k$ a non-multiplicative
variable of it. From the definition of involutive bases, there exists a
unique $j$ such that $\lm(h_j)|x_k\lm(h_i)$. We order the elements of $H$
in such a way that $i<j$ (which is always possible for a continuous
division \cite[Lemma 5.5]{Seiler2}).  Then we find a unique involutive
standard representation $x_k h_i = \sum_{j=1}^t p_j^{(i,k)} h_j$ where
$p_j^{(i,k)} \in \KK [M_{\mathcal{L},H,\prec}(h_j)]$ and the corresponding
syzygy
$ \mathbf{S}_{i,k} =x_k \mathbf{e}_i - \sum_{j=1}^t p_j^{(i,k)}
\mathbf{e}_j \in \P^t$. We denote the set of all thus obtained syzygies by
$H_{\syz} = \{ \mathbf{S}_{i ,k} \ | \ 1 \leq i \leq t; x_k \in
\NM_{\mathcal{L},H,\prec} (h_i)\}$. An involutive division
$ \mathcal{L} $ is of {\em Schreyer type} if all sets
$NM _{\mathcal{L},H,\prec}(h)$ with $h \in H$ are again involutive bases
for the ideals defined by them. Both the Janet and the Pommaret divisions
are of Schreyer type.
\begin{theorem}(\cite[Thm. 5.10]{Seiler2})
  \label{thm1}
With the above notations, let $ \mathcal{L} $ be a continuous involutive division of Schreyer type w.r.t. $ \prec $ and $H$ an  involutive basis. Then $H_{\syz}$ is an $ \mathcal{L} $-involutive basis for $\syz(H)$  w.r.t. $\prec_s$.
\end{theorem}

We now present a non-trivial variant of Gerdt's algorithm \cite{gerdt}
computing simultaneously a minimal involutive basis for the input ideal and
an involutive basis for the syzygy module of this basis. It uses an
analogous idea as the algorithm given in \cite{casc}. However, since we aim
at determining also a syzygy module, we must save the traces of all
reductions and for this reason we cannot use the syzygies to remove useless
reductions.

\begin{algorithm}[H]
\caption{{\sc InvBasis}}
\begin{algorithmic}[1]
 \REQUIRE A finite set  $ F\subset \P$; an involutive division $ \mathcal{L} $; a monomial ordering $ \prec $ 
 \ENSURE A minimal $ \mathcal{L} $-basis for $ \langle F \rangle $ and an $ \mathcal{L} $-basis for syzygy module of this basis.
 \STATE $F:=$sort$(F,\prec)$
 \STATE$T:= \lbrace  (F[1],F[1], \emptyset ,\textbf{e}_{1},false)  \rbrace$ 
 \STATE$Q:= \lbrace (F[i],F[i], \emptyset ,\textbf{e}_{i},false)  \ | \ i=2, \ldots ,\vert F \vert \rbrace$
  \STATE $S:= \lbrace \rbrace$ and $j:=|F|$
\WHILE { $Q \neq \emptyset $}
\STATE $Q:=$sort$(Q,\prec_{s})$
\STATE select and remove $p:=Q[1]$ from $Q$
\STATE $ h:=$ {\sc InvNormalForm}$(p,T, \mathcal{L}, \prec) $
\IF { $ h[1]=0 $  }
\STATE $ S:=S \cup \{h[2]\}$ 
\ENDIF
\IF { $ h[1]=0 $  and  $ \lm( \poly(p)) = \lm(\anc(p)) $ }
\STATE $ Q:= \lbrace q \in Q\,\, \vert \,\, \anc(q) \neq \poly(p) \ {\rm or} \ q[5]=true \rbrace $
\ENDIF
\IF {$p[5]=true$}
\STATE $q:=${\sc Update}$(q,p)$ for each $q\in T$
\ENDIF
\IF { $ h[1] \neq 0 $  and $ \lm(\poly(p)) \neq \lm(h) $ }
\FOR {$ q \in T  $  with proper conventional division $ \lm(h[1]) \mid \lm(\poly(q)) $}
\STATE $ Q:=Q \cup \lbrace [q[1],q[2],q[3],q[4],true] \rbrace $
\STATE $ T:=T \setminus \lbrace q \rbrace $
\ENDFOR
\STATE $j:=j+1$ and $ T:=T \cup \lbrace ( h[1],h[1],\emptyset,\textbf{e}_{j},false )  \rbrace $ 
\ELSE
\STATE $ T:=T \cup \lbrace ( h[1] ,\anc(p),\NM(p),h[2],false)  \rbrace $ 
\ENDIF
\FOR { $ q \in T  $ and $  x \in NM_{ \mathcal{L}}( \lm(\poly(q)), \lm(\poly(T)) \setminus \NM(q))$}
\STATE $ Q:=Q \cup  \lbrace (x. \poly(q),\anc(q),\emptyset,x.\rep(q),false) \rbrace    $
\STATE { \small $\NM(q):=\NM(q)  \cup NM_{\mathcal{L}}(\lm( \poly(q)),\lm( \poly(T))) \cup \lbrace x \rbrace$}
\ENDFOR
\ENDWHILE 
\STATE  {\bf{return}} $(\poly(T), \{\rep(p)-\textbf{e}_{{\rm index}(p)} \ | \ p\in T\}\cup S)$
\end{algorithmic}
\end{algorithm} 

The algorithm \textsc{InvBas} relies on the following data structure for
polynomials. To each polynomial $f$, we associate a quintuple
$ p = ( f, g, V,{\bf q},flag ) $.  The first entry $f = \poly(p)$ is the
polynomial itself, $g=\anc(p)$ is the ancestor of $f$ (realised as a
pointer to the quintuple associated with the ancestor) and $V=\NM (p)$ is
its list of already processed non-multiplicative variables. The fourth
entry ${\bf q}=\rep(p) $ denotes the representation of $f$ in our current
basis, i.e. if
${\bf q}=\sum_{r\in T\cup Q}{h_r\textbf{e}_{{\rm index}(r)}}$ then
$f=\sum_{r\in T\cup Q}{h_r\poly(r)}$ where $h_r\in \P$ and ${\rm index}(r)$
gives the position of $r$ in the current list $T\cup Q$. The final entry is
a boolean flag. If $flag=true$ then at some stage of the algorithm $p$ has
been moved from $T$ to $Q$, otherwise $flag=false$.  We denote by
$\sig(p)=\lm_{\prec_s}(\rep(p))$ the signature of $p$. By an abuse of
notation, $\sig(f)$ also denotes $\sig(p)$. The same holds for the $\rep$
function. If $P$ is a set of quintuples, we denote by $ \poly(P)$ the set
$ \lbrace \poly(p) \ | \ p \in P \rbrace$. In addition, the functions
sort($X,\prec$) and sort($ X,\prec _{s} $) sort $X$ in increasing order
according to $ \lm(X)$ w.r.t. $ \prec $ and $\{\sig(p)\ \vert \ p\in X\}$
w.r.t. $\prec_s$, respectively. We remark that in the original form of
Gerdt's algorithm \cite{gerdt} the function sort($Q,\prec$) was applied to
sort the set of all non-multiplicative prolongations, however, in our
experiments we observed that using sort($Q,\prec_s$) increased the
performance of the algorithm.

Obviously, the representation of each polynomial must be updated whenever
the set $T\cup Q$ changes in a non-trivial way.  We remark that elements of
$Q$ can appear non-trivially in the representations of polynomials only if
they have been elements of $T$ at an earlier stage of the algorithm (recall
that such a move is noted in the flag of each quintuple), as all reductions
are performed w.r.t. $T$ only.  If updates are necessary, then they are
performed by the function \textsc{Update}.  Involutive normal forms are
computed with the help of the following subalgorithm taking care of the representations.

 \begin{algorithm}[H]
\caption{{\sc InvNormalForm}}
\begin{algorithmic}
 \REQUIRE A quintuple $ p $; a set of quintuples $ T $; a division $ \mathcal{L} $; a monomial ordering $ \prec $ 
 \ENSURE A normal form of $ p $ w.r.t. $ T $ and its new representation.
  \STATE $ h := \poly(p)  $ and  $G := \poly(T) $ and ${\bf q}:=  \rep(p)$
\WHILE {  $h$ contains a monomial $m$ which is $ \mathcal{L}$-divisible by $g \in G$ }
\IF { $ m = \lm(\poly(p)) $  and   {\sc C1}$(h, g)$}
\STATE
 {\bf{return}} $ ([0,\anc(p) \rep(\anc(g))-\anc(g) \rep(\anc(p))]) $
\ENDIF
\STATE $ h := h - (cm/\lt(g)).g $ where $c$ is the coefficient of $m$ in $h$
\STATE $ {\bf q}:={\bf q} -( cm/\lt(g)) \rep(g)$
\ENDWHILE 
\STATE  {\bf{return}} $( [h,{\bf q}])$
\end{algorithmic}
\end{algorithm}

Here we apply the involutive form of Buchberger's first criterion
\cite{gerdt}. We say that {\sc C1}$(p, g)$ is true if
$\lm(\anc(p))\lm(\anc(g)) = \lm(\poly(p))$. 

\begin{theorem}\label{thm2}
  If $ \mathcal{L} $ is a Noetherian continuous involutive division of
  Schreyer type then {\sc InvBasis} terminates in finitely many steps and
  returns a minimal involutive basis for its input ideal and also an
  involutive basis for the syzygy module of the constructed basis.
\end{theorem}

\begin{proof}
  The termination of the algorithm is ensured by the termination of Gerdt's
  algorithm, see \cite{gerdt}. Let us now deal with its correctness. We
  first note that if an element $p$ is removed by Buchberger's criteria,
  then it is superfluous and by \cite[Thm.~2]{gerdt} the set $\poly(T)$
  forms a minimal involutive basis for $\langle F \rangle$. Thus, it
  remains to show that
  $R=\{\rep(p)-\textbf{e}_{{\rm index}(p)} \ | \ p\in T\}\cup S$ is an
  involutive basis for $\poly(T)=\{h_1,\ldots, h_t\}$
  w.r.t. $\prec_s$. Using Thm. \ref{thm1}, we must show that the
  representation of each non-multiplicative prolongation of the elements of
  $\poly(T)$ appears in $R$. Let us consider $h_i\in \poly(T)$ and a
  non-multiplicative variable $x_k$ for it. Then, due to the structure of
  the algorithm, $x_kh_i$ is created and studied in the course of the
  algorithm.

  Now, four cases can occur.  If $x_kh_i$ reduces to zero then we can write
  $x_k h_i = \sum_{j=1}^t p_j^{(i,k)} h_j$ where
  $p_j^{(i,k)} \in \KK [M_{\mathcal{L},H,\prec}(h_j)]$. Therefore the
  representation
  $x_k \mathbf{e}_i - \sum_{j=1}^t p_j^{(i,k)} \mathbf{e}_j \in \P^t$ is
  added to $S$ and consequently it appears in $R$. If the involutive normal
  form of $x_kh_i$ is non-zero then we can write
  $x_k h_i = \sum_{j=1}^t p_j^{(i,k)} h_j+h_\ell$ where
  $p_j^{(i,k)} \in \KK [M_{\mathcal{L},H,\prec}(h_j)]$. In this case, we
  add $h_\ell$ into $T$ and the representation component of $x_k h_i$ is
  updated to $x_k \mathbf{e}_i - \sum_{j=1}^t p_j^{(i,k)}
  \mathbf{e}_j$. Then, as we can see in the output of the algorithm,
  $x_k \mathbf{e}_i - \sum_{j=1}^t p_j^{(i,k)}
  \mathbf{e}_j-\mathbf{e}_\ell$ appears in $R$ as the syzygy correcponding
  to $x_k h_i$.

  The third case that may occur is that $x_kh_i$ is removed by Buchberger's
  first criterion. Assume that $p$ is the quintuple associated to $x_k h_i$
  and $g$ is another quintuple so that {\sc C1}$(p, g)$ is true. It follows
  that $\lm(\anc(p))\lm(\anc(g)) = \lm(\poly(p))$ holds. We may let
  $x_k h_i=u \anc(p)$, $\poly(g)=v \anc(g)$ and $\lm(x_k h_i)=m \lm(g)$ for
  some monomials $u$ and $v$ and term $m$ (assume that the polynomials are
  monic). Thus,
\begin{eqnarray*}
    x_k h_i -m\poly(g)=u \anc(p)-m v \anc(g).
\end{eqnarray*}
As $\lm(\anc(p))\lm(\anc(g))=\lcm(\lm(\anc(p)), \lm(\anc(g)))$,
Buchberger's first criterion applied to $\anc(p)$ and $\anc(g)$ yields that
$\anc(p) \rep(\anc(g))-\anc(g) \rep(\anc(p))$ is the corresponding syzygy
which is added to $S$.

The last case to be considered is that $x_kh_i$ is removed by the second
{\bf if}-loop in the main algorithm. In this case, we conclude that
$\anc(p)$ is reduced to zero and in consequence $h_i$ is reduced to
zero. So, $h_i$ is a useless polynomial and we do not need to keep $x_kh_i$
which ends the proof.$\qed$
\end{proof}

\begin{remark}
  There also exists an involutive version of Buchberger's second criterion
  \cite{gerdt}: {\sc C2}$(p, g) $ is true if
  $\lcm(\lm(\anc(p)), \lm(\anc(g)))$ properly divides $\lm(\poly(p))$. We
  cannot use this criterion in the {\sc InvNormalForm} algorithm. A
  non-multiplicative prolongation $x_kh_i$ removed by it is surely useless
  in the sense that is not needed for determining the involutive basis of
  $\I$, but it can nevertheless be necessary for the construction of its
  syzygy module.
\end{remark}

\begin{example}
  Let us consider the ideal $\I$ generated by $ F= \lbrace f_1=z^2, f_2=zy , f_3=xz-y, f_4=y^2 , f_5=xy-y, f_6=x^2-x+z \rbrace  \subset \KK[x,y,z]  $ from \cite[Ex. 5.6]{Seiler2}. Then,  $F$  is a Janet basis w.r.t. $z \prec y \prec x$. Since $x,y$ are non-multiplicative variables for $ f_1,f_2,f_3 $ and $x $ is non-multiplicative variable  for $f_4,f_5 $ then the following set is a Janet basis for the syzygy module of $F$:
 $ \{ y\mathbf{e}_1-z\mathbf{e}_2,x\mathbf{e}_1-z\mathbf{e}_3-\mathbf{e}_2 , y\mathbf{e}_2-z\mathbf{e}_4, x\mathbf{e}_2-z\mathbf{e}_5-\mathbf{e}_2 , y\mathbf{e}_3-z\mathbf{e}_5+\mathbf{e}_4-\mathbf{e}_2 , x\mathbf{e}_3-z\mathbf{e}_6+\mathbf{e}_5-\mathbf{e}_3+\mathbf{e}_1 , x\mathbf{e}_4-y\mathbf{e}_5-\mathbf{e}_4, x\mathbf{e}_5-y\mathbf{e}_6+\mathbf{e}_2 \} $.
\end{example}

\section{Application to Pommaret Basis Computation}
\label{Sec:3}

In this section we show how to apply the approach presented in the
preceeding section in the computation of Pommaret bases. The Pommaret
division is not Noetherian and thus a given ideal may not have a finite
Pommaret basis. However, a generic linear change of variables transforms
the ideal into quasi stable position where a finite Pommaret basis

exists. Seiler \cite{Seiler2} proposed a deterministic algorithm to compute
such a linear change by performing repeatedly an {\em elementary} linear
change and then a test on the Janet basis of the transformed ideal. Now, to
apply the method presented in this paper, we use the {\sc InvBasis}
algorithm to compute a minimal Janet basis $H$ for the input ideal and at
the same time a Janet basis for $\syz(H)$. Then, for each $h\in H$ we check
whether there exists a variable which is Janet but not Pommaret
multiplicative. If not, $H$ is a Pommaret basis and we are done. Otherwise,
we make an elementary linear change of variables, say $\phi$. Then, we
apply the following algorithm, {\sc NextInvBasis}, to compute a minimal
Janet basis for the ideal generated by $\phi(H)$ by applying
$\phi(\syz(H))$ to remove superfluous reductions. We describe first the
main procedure.

\begin{algorithm}[H]
\caption{{\sc QuasiStable}}
\begin{algorithmic}
 \REQUIRE  A finite set $F \subset \P$ of homogeneous polynomials and a monomial ordering $\prec$ 
 \ENSURE A linear change $\Phi$ so that $\langle \Phi(F) \rangle$ has a finite Pommaret basis
 \STATE $\Phi:=$the identity map
 \STATE $J,S:=${\sc InvBasis}$(F,\mathcal{J},\prec)$ and $ A:=${\sc Test}$(\lm(J)) $
 \WHILE  {$ A \neq true $} 
 \STATE $\phi:=A[3]\mapsto A[3]+cA[2]$ for a random choice of $c\in \KK$
 \STATE $Temp:=${\sc NextInvBasis}$(\Phi \circ \phi (J),\Phi\circ \phi (S),\mathcal{J},\prec)$
 \STATE $ B:=${\sc Test}$(\lm(Temp)) $
 \IF  {$ B \neq A $} 
 \STATE  $\Phi:=\Phi\circ\phi$ and  $A := B$
 \ENDIF
 \ENDWHILE
\STATE  {\bf{return}} $(\Phi)$
\end{algorithmic}
\end{algorithm} 

The function {\sc Test} receives a set of monomials forming a minimal Janet
basis and returns true if it is a Pommaret basis, too. Otherwise,, by
\cite[Prop.~2.10]{Seiler2}, there exists a monomial $m$ in the set for
which a Janet multiplicative variable (say $x_\ell$) is not Pommaret
multiplicative. In this case, the function returns
$(false,x_\ell,\cls(m))$. Using these variables, we construct an elementary
linear change of variables.

The \textsc{NextInvBasis} algorithm is similar to the {\sc InvBasis}
algorithm given above. However, the new algorithm computes only the
involutive basis of the input ideal generated by a set $H$. In addition, in
the new algorithm, we use $\syz(H)$ to remove useless reductions. Below,
only the differences between the two algorithms are exhibited.

\begin{algorithm}[H]
\caption{{\sc NextInvBasis}}
\begin{algorithmic}
 \REQUIRE A finite set $ F\subset \P $; a generating set $S$ for $\syz(F)$; an involutive division $ \mathcal{L} $; a monomial ordering $ \prec $
 \ENSURE A minimal involutive basis for $ \langle F \rangle $
 \STATE $\vdots$\qquad  \COMMENT{Lines 1--6 of  \textsc{InvBasis}}
 \STATE select and remove $p:=Q[1]$ from $Q$
\IF{ $\nexists {\bf s} \in S $  s.t  $  \lm_{\prec_s}({\bf s}) \mid \sig(p) $ }
\STATE $\vdots$\qquad  \COMMENT{Lines 8--30 of  \textsc{InvBasis}}
\ENDIF
\STATE $\vdots\qquad $ \COMMENT{Lines 31/32 of \textsc{InvBasis}}
\end{algorithmic}
\end{algorithm}

\begin{lemma}
  Let $H\subset \P$ and $S$ be a generating set for $\syz(H)$. For any invertible linear change of variables $\phi$, $\phi(S)$ generates $\syz(\phi(H))$.
\end{lemma}
\begin{proof}
 Suppose that  $H=\{h_1,\ldots ,h_t\}$ and  $S=\{{\bf s_1},\ldots ,{\bf s_\ell}\}\subset \P^t$. Let ${\bf s_i}=(p_{i1},\ldots ,p_{it})$. Since $p_{i1}h_1+\cdots +p_{it}h_t=0$ and $\phi$ is a ring homomorphism then $\phi(p_{i1})\phi(h_1)+\cdots +\phi(p_{it})\phi(h_t)=0$ and therefore $\phi({\bf s_i})\in \syz(\phi(H))$. Conversely, assume that ${\bf s}=(p_1,\ldots ,p_t)\in \syz(\phi(H))$. This shows that $p_{1}\phi(h_1)+\cdots +p_{t}\phi(h_t)=0$. By invertibility of $\phi$  we have $(\phi^{-1}(p_1),\ldots ,\phi^{-1}(p_t))\in \syz(H)$. From assumptions, we conclude that $(\phi^{-1}(p_1),\ldots ,\phi^{-1}(p_t))=g_1{\bf s_1}+\dots +g_\ell{\bf s_\ell}$ for some $g_i\in \P$. By applying $\phi$ on both sides of this equality, we can deduce that ${\bf s}$ is generated by $\phi(S)$ and the proof is complete. \qed
\end{proof}

\begin{theorem}
  The algorithm {\sc QuasiStable} terminates in finitely many steps and
  returns for a given homogeneous ideal a linear change of variables
  s.t. the transformed ideal possesses a finite Pommaret basis.
\end{theorem}
\begin{proof}
  Seiler \cite[Prop. 2.9]{Seiler2} proved that for a generic linear change
  of variables $\phi$, the ideal $\langle \phi(F) \rangle$ has a finite
  Pommaret basis. He also showed that the process of finding such a linear
  change, by applying elementary linear changes, terminates in finitely
  many steps, see \cite[Remark~9.11]{Seiler2} (or \cite{hss:detgen}). These
  arguments establish the finite termination of the algorithm. To prove the
  correctness, using Thm.~\ref{thm2}, we must only show that if $p\in Q$ is
  removed by ${\bf s}\in S$ then it is superfluous. To this end, assume
  that $F=\{f_1,\ldots ,f_k\}$ and ${\bf s}=(p_1,\ldots ,p_k)$. Thus, we
  have $p_1f_1+ \cdots +p_kf_k=0$. On the other hand, we know that
  $ \lm_{\prec_s}({\bf s}) \mid \sig(p) $. W.l.o.g., we may assume that
  $\lm_{\prec_s}({\bf s})=\lm(p_1){\bf e_1}$. Therefore, $\poly(p)$ can be
  written as a combination $g_1f_1+ \cdots +g_kf_k$ such that $\lm(g_1)$
  divides $\lm(p_1)$. Let $t=\lm(p_1)/\lm(g_1)$. We can write $\lm(g_1)f_1$
  as a linear combination of some multiplications $mf_i$ where $m$ is a
  monomial such that $m{\bf e_i}$ is strictly smaller than
  $\lm(g_1){\bf e_1}$.  It follows that $p$ has an involutive
  representation provided that we study $tmf_i$ for each $m$ and $i$. Since
  the signature of $tmf_i$ is strictly smaller than
  $t\lm(g_1){\bf e_1}=\sig(p)$, we are sure that no loop is performed
  and therefore $p$ can be omitted.  \qed
\end{proof}

We have implemented the algorithm {\sc QuasiStable} in {\sc Maple
  17}\footnote{The {\sc Maple} code of the implementations of our
  algorithms and examples are available at {\tt
    http://amirhashemi.iut.ac.ir/softwares}} and compared its performance
with our implementation of the {\sc HDQuasiStable} algorithm presented in
\cite{casc} (it is a similar procedure applying a Hilbert driven
technique).  For this, we used some well-known examples from computer
algebra literature. All computations were done over $\mathbb{Q}$ using the
degree reverse lexicographical monomial ordering. The results are
represented in the following tables where the time and memory columns
indicate the consumed CPU time in second and amount of megabytes of used
memory, respectively. The dim column refers to the dimension of the
corresponding ideal. The columns corresponding to $C_1$ and $C_2$ show,
respectively, the number of polynomials removed by $ C_{1} $ and $ C_{2} $
criteria. The seventh column denotes the number of polynomials eliminated
by the criterion related to signature applied in {\sc NextInvBasis}
algorithm (see \cite{casc} for more details).  The eighth column shows the
number of polynomials eliminated by the Hilbert driven technique which may
be applied in {\sc NextInvBasis} algorithm to remove useless reductions,
(see \cite{casc} for more details). The ninth column shows the number of
polynomials eliminated by the syzygy criterion described in {\sc
  NextInvBasis} algorithm.  The last three columns represent, respectively,
the number of reductions to zero, the number of performed elementary linear
changes and the maximum degree attained in the computations. The
computations in this paper are performed on a personal computer with $2.60$
GHz Pentium(R) Core(TM) Dual-Core CPU, $2$ GB of RAM, $32$ bits under the
Windows $7$ operating system.

\begin{center}
{\scriptsize
\begin{tabular}{ |c||c|c|c|c|c|c|c|c|c|c|c| }

\multicolumn{1}{c}{}  \\
\hline
Weispfenning94 & time & memory& dim & $C_{1}$ & $C_{2}$ & SC  & HD & $\small{ \syz }$  & redz & lin &deg  \\
\hline
{\sc QuasiStable}  & 4.5 & 255.5 & 2 & 0  & 0 & 0 & 34 & 10 & 41 & 1 &14  \\
\hline 
{\sc HDQuasiStable} & 5.3 & 261.4 & 2 & 0 & 1 & 9 & 46 & - & 29 & 1 & 14  \\
\hline

\multicolumn{1}{c}{}  \\
\hline
Liu & time & memory & dim & $C_{1}$ & $C_{2}$ & SC  & HD & $\small{ \syz }$  & redz & lin &deg  \\
\hline
{\sc QuasiStable}  & 6.1 & 246.7 & 2 & 8 & 0 & 10 & 71 & 47  & 44  & 4 & 6  \\
\hline 
{\sc HDQuasiStable}  & 8.9 & 346.0 & 2  & 6 & 3 & 25  & 125 & - & 60  & 4 & 6  \\
\hline

\multicolumn{1}{c}{}  \\
\hline
Noon & time & memory & dim &$C_{1}$ & $C_{2}$ & SC  & HD & $\small{ \syz }$  & redz & lin &deg  \\
\hline
{\sc QuasiStable} & 74.1 & 3653.2.2 & 1 & 6 & 7 & 10  & 213 & 83 &  215  &  4 & 10 \\
\hline 
{\sc HDQuasiStable} & 72.3 & 3216.9.7 & 1 & 4 & 24  & 10 & 351 & -  & 105  &  4  &  10 \\
\hline

\multicolumn{1}{c}{}  \\
\hline
Katsura5 & time & memory & dim & $C_{1}$& $C_{2}$ & SC  & HD & $\small{ \syz }$  & redz & lin &deg  \\
\hline
{\sc QuasiStable}  & 95.7 & 4719.2 & 5& 49 & 0 & 0 & 257  & 56 & 115 & 3 & 8 \\
\hline 
{\sc HDQuasiStable}  & 120.8 & 5527.7 & 5 & 44 & 4 & 6 &  420   & - & 122 & 3 & 8 \\
\hline

\multicolumn{1}{c}{}  \\
\hline
Vermeer & time & memory & dim & $C_{1}$ & $C_{2}$ & SC  & HD & $\small{ \syz }$  & redz & lin &deg  \\
\hline
{\sc QuasiStable}  &175.5 & 8227.9  & 3 & 5 & 3 &101 &158  & 139& 343 & 3 & 13  \\
\hline 
{\sc HDQuasiStable}  &  192.5 & 8243.7 & 3 & 3 & 28 & 157 &  343 & - & 190 & 3 & 13  \\
\hline

\multicolumn{1}{c}{}  \\
\hline
Butcher & time & memory & dim &$C_{1}$ & $C_{2}$ & SC  & HD & $\small{ \syz }$  & redz & lin &deg  \\
\hline
{\sc QuasiStable}  & 290.6 & 12957.8 & 3 & 135 & 89 & 73 & 183 & 86 & 534 & 3 & 8 \\
\hline 
{\sc HDQuasiStable}  & 433.1 & 17005.5 & 3 & 178 & 178 & 219 &  355 & - & 386 & 3 & 8  \\
\hline

\end{tabular} 
}
\end{center}

As one sees for some examples, some columns are different. It is worth noting that this difference may be due to the fact that the coefficients in the linear changes  are chosen randomly and this may affect the behavior of the algorithm.
\section{Involutive Variant of the GVW Algorithm}
\label{Sec:4}
Gao et al.\  \cite{gvw} described  recently a new algorithm,  the GVW
algorithm, to compute simultaneously Gr\"obner bases for a given ideal and
for the syzygy module of the given ideal basis. In this section, we present an involutive variant of this approach and compare its efficiency with the existing algorithms to compute involutive bases. For a review of the general setting of the signature based structure that we use in this paper, we refer to \cite{gvw}. Let $ \{ f_1,\dots,f_k \} \subset \P $ be a finite set of non-zero polynomials and $\{{\bf e}_{1}, \ldots ,{\bf e}_{k} \}$  the standard basis for $ \P^{k} $. Let us fix an involutive division $\mathcal{L}$ and a monomial ordering $\prec$. Our goal is  to compute an involutive basis for $ \I= \langle f_{1}, \ldots ,f_{k} \rangle  $ and a Gr\"obner basis for $ \syz( f_{1}, \ldots ,f_{k}) $ w.r.t. $\prec_s$. Let us consider
\[ \mathcal{V}= \{ ({\bf u},v)\in \P^k \times \P \ \vert \  u_1 f_1 + \cdots + u_k f_k =v  \ {\rm with} \ {\bf u}=(u_1,\ldots ,u_k)\} \] 
as an $ \P $-submodule  of $ \P^{k+1} $. For any pair $p = ({\bf u},v) \in \P^k \times \P$, $\lm_{\prec_s}({\bf u})$ is called the {\em signature} of $p$ and is denoted by $\sig(p)$.  We define the involutive version of top-reduction defined in \cite{gvw}.
 Let $p_1 = ({\bf u}_1 , v_1 )$, $ p_2 =({\bf u}_2 , v_2 ) \in \P^k \times \P$. When $v_2$ is non-zero, we say $p_1$ is {\em involutively top-reducible} by $p_2$ if:
\begin{itemize}
\item
$v_1 $ is non-zero and $\lm(v_2 )$ $\mathcal{L}$-divides $\lm(v_1 )$ and
\item
$\lm(t{\bf u}_2 ) \preceq_{s} \lm({\bf u}_1 )$ where $t = \lm(v_1 )/  \lm(v_2 )$.
\end{itemize}
  The corresponding top-reduction is $p_1 - ctp_2 = ({\bf u}_1 - ct{\bf
    u}_2 , v_1 - ctv_2 )$ where $ c = \lc(v_1 ) / \lc(v_2 ) $. Such a
  top-reduction is called {\em regular}, if $\lm({\bf u}_1 -ct{\bf u}_2) =
  \lm({\bf u}_1)$, and {\em super} otherwise.
  
\begin{definition}
A finite subset  $G\subset  \mathcal{V}$ is called  a {\em strong
  involutive basis} for $  \mathcal{I}$ if every pair in $ \mathcal{V}$ is
involutively top-reducible by some pair in $G$.  A strong
  involutive basis  $G$ is {\em minimal} if any other strong
  involutive basis $G'$ of $  \mathcal{I}$  satisfies $\lm(G)\subseteq \lm(G')$.
\end{definition}

\begin{proposition}
Suppose that $G = \{({\bf u}_1 , v_1 ), \ldots , ({\bf u}_m , v_m )\}$ is a strong involutive basis for $ \mathcal{I}$. Then $G_{0} = \{{\bf u}_{i} \ \vert \  v_i = 0 \ , \ 1 \leq  i \leq m \} $ is a Gr\"obner basis for $\syz(f_1 , \ldots , f_k )$, and $G_1 = \{v_1,\ldots ,v_m \}$ is an involutive basis for $\I$.
\end{proposition}

\begin{proof}
The proof is an easy consequence of the proof of  \cite[Prop. 2.2]{gvw}. \qed
\end{proof}

Let $p_{1 }= ({\bf u}_{1} , v_{1} )$ and $p_{2} = ({\bf u}_{2} , v_{2} )$
be two pairs in $ \mathcal{V}$. We say that $p_{1}$ is {\em covered} by
$p_{2}$ if $\lm({\bf u}_{2} )$ divides $\lm({\bf u}_{1} )$ and
$t \lm(v_{2} ) \prec \lm(v_{1} )$ (strictly smaller) where
$t = \lm({\bf u}_{1} ) / \lm({\bf u}_{2} )$. Also, $p$ is covered by $G$ if
it is covered by some pair in $G$. A pair $p\in \mathcal{V}$ is {\em
  eventually super reducible} by $G$ if there is a sequence of regular
top-reductions of $p$ by $G$ leading to $({\bf u}', v')$ which is no longer
regularly reducible by $G$ but super reducible by $G$.
  
\begin{theorem}
\label{thm4}
  Let  $G \subset  \mathcal{V}$ be a finite set such that, for any module monomial ${\bf m} \in  \P^{k}$, there is a pair $({\bf u}, v) \in  G$  such that $ \lm({\bf u}) \mid {\bf m}$. Then the following conditions are equivalent:
\begin{enumerate}
\item
$G$ is a strong involutive basis for $ \I$,
\item
any non-multiplicative prolongation of any element of $G$ is eventually super top-reducible by $G$,
\item
any non-multiplicative prolongation of any element in $G$ is covered by $G$.
\end{enumerate}
\end{theorem}

\begin{proof}
  The proof of all implications are similar to the proofs of the
  corresponding statements in \cite[Thm.~2.4]{gvw} except that we need some
  slight changes in the proof of $( 3 \Rightarrow 1) $. We proceed by
  reductio ad absurdum. Assume that there is a pair
  $p=({\bf u} , v )\in \mathcal{V}$ which is not involutively top-reducible
  by $G$ and has minimal signature. Then, by assumption, there exists
  $p_1 = ({\bf u}_1, v_1) \in G$ such that $\lm({\bf u})=t\lm({\bf u}_1)$
  for some $t$. Select $p_1$ such that $t\lm(v_1)$ is minimal. Let us now
  consider $tp_1$. Two cases may happen: If all variables in $t$ are
  multiplicative for $p_1$ then, $p-tp_1$ has a signature smaller than $p$
  and by assumption it has a standard representation leading to a standard
  representation for $p$ which is a contradiction. Otherwise, $t$ has a
  non-multiplicative variable. Then, $t p_1$ is covered by a pair
  $p_3 = ({\bf u}_3, v_3) \in G$. This shows that
  $t_3 \lm(v_3)\prec t \lm(v_1)$ with $t_3=t\lm({\bf u}_1)/\lm({\bf
    u}_3)$. Therefore, the polynomial part of $t_3p_3$ is smaller than
  $tv_1$ which contradicts the choice of $p_1$, and this ends the
  proof. \qed
\end{proof}

 Based on this theorem and similar to the structure of the GVW algorithm,
 we describe a variant of Gerdt's algorithm for computing strong involutive
 bases. The structure of the new algorithm is similar to the {\sc InvBasis} algorithm and therefore we omit the identical parts.

 \begin{algorithm}[H]
\caption{{\sc StInvBasis}}
\begin{algorithmic}
 \REQUIRE A finite set  $ F\subset \P$; an involutive division $ \mathcal{L} $; a monomial ordering $ \prec $ 
 \ENSURE A minimal strong involutive basis for $ \langle F \rangle $ 
 \STATE $F:=$sort$(F,\prec)$ and $T:= \lbrace  (F[1],F[1], \emptyset ,\textbf{e}_{1}) \rbrace$
 \STATE $Q:= \lbrace (F[i],F[i], \emptyset ,\textbf{e}_{i})  \ | \ i=2, \ldots ,\vert F \vert \rbrace$ and  $H:= \lbrace \rbrace$
\WHILE { $Q \neq \emptyset $}
\STATE $Q:=$sort$(Q,\prec_{s})$ and select/remove the first element $p$ from $Q$
\IF { $ p $ is not covered by $ G $, $ T $ or $ H $}
\STATE $ h:=$ {\sc InvTopReduce}$(p,T, \mathcal{L}, \prec) $
\IF {$ \poly(h)=0 $ }
\STATE $ H:= H \cup \{ \sig(p) \}$
\ENDIF
\IF{ $ \poly(h)=0 $  and  $ \lm( \poly(p)) = \lm(\anc(p)) $ }
\STATE $ Q:= \lbrace q \in Q\,\, \vert \,\, \anc(q) \neq \poly(p) \rbrace $
\ENDIF
\IF { $ \poly(h) \neq 0 $  and $ \lm(\poly(p)) \neq \lm(\poly(h)) $ }
\STATE $\vdots$\qquad  \COMMENT{Lines 19--25 of \textsc{InvBas}}
\ENDIF
\STATE $\vdots$\qquad  \COMMENT{Lines 27--30 of \textsc{InvBas}}
\ENDIF
\ENDWHILE 
\STATE  {\bf{return}} $(\poly(T),H)$
\end{algorithmic}
\end{algorithm} 
 \begin{algorithm}[H]
\caption{{\sc InvTopReduce}}
\begin{algorithmic}
 \STATE {\bf{Input:}} A quadruple $ p $; a set of quadruples $ T $; a division $ \mathcal{L} $; a monomial ordering $ \prec $ 
 \STATE {\bf{Output:}} A top-reduced form of $ p $ modulo $ T $
  \STATE  $ h := p $ 
\WHILE {  $\poly(h)$ has a term $am$ with $a\in \KK$ and $\lm(\poly(q))\mid_{\mathcal{L}} m$ with $q\in T$}
\IF {  $ m/\lm(\poly(q))\sig(q) \prec_{s} \sig(p) $ }
\STATE $ \poly(h) := \poly(h) - am/ \lt(\poly(q)).\poly(q) $
\STATE $\rep(h):=\rep(h) - am/ \lt(\poly(q)).\rep(q)$
\ENDIF
\ENDWHILE 
\RETURN $(h)$
\end{algorithmic}
\end{algorithm}
\vspace*{-0.5cm}
The proof of the next theorem is a consequence of Thm. \ref{thm4} and the termination and correctness of Gerdt's algorithm. 
 \begin{theorem}
   If $ \mathcal{L} $ is Noetherian, then {\sc StInvBasis} terminates in
   finitely many steps returning a minimal strong involutive basis for its
   input ideal.
\end{theorem}

We have implemented the {\sc StInvBasis} algorithm in {\sc Maple 17} and compared its performance with our implementation of {\sc InvolutiveBasis} algorithm (see \cite{casc}) and {\sc VarGerdt} algorithm (a variant of Gerdt's algorithm, see \cite{zbMATH06264259}).
\begin{center}
{\scriptsize 
\begin{tabular}{ |c||c|c|c|c|c|c|c|c| } 
\hline
Liu & time & memory & $C_{1}$ & $C_{2}$ &  SC  &  cover   & redz & deg \\
\hline
{\sc StInvBasis}  & .390 & 14.806 & - & - & - & 17 & 20 & 6  \\
\hline 
{\sc InvolutiveBasis}& .748 & 23.830 & 4 & 3 & 2  & - & 18 & 6  \\
\hline
{\sc vargerdt}  & 1.653 & 64.877 & 6 & 3 & - & - & 18 & 19  \\
\hline

\multicolumn{1}{c}{}  \\
\hline
Noon & time & memory & $C_{1}$ & $C_{2}$ & SC &  cover   & redz & deg \\
\hline
{\sc StInvBasis}  & 1.870 & 75.213 & - & - & - & 54 & 42 & 10 \\
\hline 
{\sc InvolutiveBasis} & 2.620 & 105.641 & 4 & 15 & 6  & - & 50 & 10 \\
\hline 
{\sc vargerdt}  & 12.32 & 454.573 & 6 & 9 & -  & - & 56 & 10  \\
\hline

\multicolumn{1}{c}{}  \\
  \hline
Haas3 & time & memory & $C_{1}$ & $C_{2}$ & SC & cover   & redz & deg \\
\hline
{\sc StInvBasis}  & 157.623 & 6354.493 & - & - & - & 490 & 8 & 33   \\
\hline 
{\sc InvolutiveBasis}  & 22.345 & 833.0 & 0 & 0 & 83  &- & 152 & 33  \\
\hline
{\sc vargerdt}  & 137.733 & 5032.295 & 0 & 98 & -  & - & 255 & 33  \\
\hline

\multicolumn{1}{c}{}  \\
\hline
Sturmfels-Eisenbud & time & memory & $C_{1}$ & $C_{2}$ & SC & cover & redz & deg \\
\hline
{\sc StInvBasis}  & 2442.414 & 120887.953 & - & - & - & 634 & 29 & 8  \\
\hline 
{\sc InvolutiveBasis}  & 24.70 & 951.070 & 28 & 103  & 95  & - & 81 & 6  \\
\hline
{\sc vargerdt}  & 59.32 & 2389.329 & 43 & 212 & -  & - & 91 & 6  \\
\hline

\multicolumn{1}{c}{}  \\
\hline
Weispfenning94 & time & memory & $C_{1}$ & $C_{2}$ & SC & cover & redz & deg \\
\hline
{\sc StInvBasis}  & 183.129 & 8287.044 & - & - & - & 588 & 28 & 18   \\
\hline 
{\sc InvolutiveBasis}  & 1.09 & 45.980 &0 & 1  & 9  & - & 28 & 10  \\
\hline
{\sc vargerdt}  & 4.305 & 168.589 & 0 & 9 & -  & - & 38 & 15  \\
\hline

\end{tabular}
}
\end{center}
As we observe, the performance of the new algorithm is not in general
better than that of the others. This is due to the signature-based structure of
the new algorithm which does not allow to perform full normal forms.

 \section*{Acknowledgments.}   The research of the second author was in part supported by a grant from IPM (No. 95550420). The work of the third author was partially performed as part of the H2020-FETOPEN-2016-2017-CSA project $SC^{2}$ (712689).

\bibliographystyle{acm}

\end{document}